\documentclass[a4paper,12pt]{amsart}
\usepackage{amssymb,latexsym,amsmath,amsthm, mathrsfs, tikz}
\usepackage[utf8]{inputenc}

\newtheorem{theorem}{Theorem}[section]

\newtheorem{proposition}[theorem]{Proposition}

\theoremstyle{definition}
\newtheorem{definition}[theorem]{Definition}

\newtheorem{example}[theorem]{Example}

\theoremstyle{remark}
\newtheorem{remark}[theorem]{Remark}

\title{Operator algebras associated with graphs and categories of paths: a Survey}
\author[Bukoski]{Juliana Bukoski}
\address{Department of Mathematics, Physics, and Computer Science, Georgetown college, 400 East College Street, Georgetown, KY, USA, 40324.}
\email{juliana\_bukoski@georgetowncollege.edu}
\author[Singla]{Sushil Singla}
\address{Department of Mathematics, Shiv Nadar University Delhi NCR, Tehsil Dadri, Greater Noida, Uttar Pradesh, India.}
\email{ss774@snu.edu.in}
\subjclass[2020]{46L05, 20L05}
\keywords{Graph $C^*$-algebras, Free semigroupoid algebras, Higher rank graphs, category of paths}
\usepackage{amssymb}
\usepackage{amsfonts}
\usepackage{tikz-cd} 
\usepackage{tikz}
\usetikzlibrary{arrows}

\begin{document}

\maketitle

\begin{abstract}
Many interesting examples of operator algebras, both self-adjoint and non-self-adjoint, can be constructed from directed graphs. In this survey, we overview the construction of $C^*$-algebras from directed graphs and from two generalizations of graphs: higher rank graphs and categories of paths. We also look at free semigroupoid algebras generated from graphs and higher rank graphs, with an emphasis on the left regular free semigroupoid algebra. We give examples of specific graphs and the algebras they generate, and we discuss properties such as semisimplicity and reflexivity. Finally, we propose a new construction: applying the left regular free semigroupoid construction to categories of paths.
\end{abstract}

\section{Introduction}\label{intro}

A directed graph is a set of vertices along with a set of edges, where each edge has a source vertex and a range vertex.  Such a graph can be represented by a collection of operators on a Hilbert space $\mathcal{H}$; each vertex is associated to a projection, and each edge is associated to a partial isometry that maps between the subspaces corresponding to its source and range vertices. These projections and partial isometries are used to construct a $C^*$-algebra called the \textit{graph $C^*$-algebra} of the directed graph. There are many examples of common $C^*$-algebras which can be realized as graph algebras, and many properties of graph algebras are determined by structural properties of the graph. 

Free semigroupoid algebras generated by directed graphs are a class of non-self-adjoint operator algebras introduced by Kribs and Power \cite{kribs.power.2004}. The construction of these algebras from a graph is similar to the graph $C^*$-algebra construction in that vertices are represented by projections and edges by partial isometries. However, a free semigroupoid algebra is closed in the weak operator topology and does not (necessarily) include adjoints. As in the graph $C^*$-algebra case, many previously-studied non-self-adjoint operator algebras can be expressed as free semigroupoid algebras for some directed graph, and many properties of the algebra correspond to properties of the graph. In fact, this relationship is in some sense stronger than the self-adjoint case; while it is possible to find two non-isomorphic graphs that produce the same graph $C^*$-algebra, Kribs and Power \cite{kribs.power.2004} showed that two free semigroupoid algebras from graphs are unitarily equivalent if and only if their corresponding graphs are isomorphic.

Both the $C^*$-algebra and the free semigroupoid algebra construction have been extended to higher rank graphs, which are a generalization of graphs where edges have length in $\mathbb{N}^k$ instead of $\mathbb{N}$; higher rank graphs can be thought of as graphs where certain paths are identified, according to a factorization property. Categories of paths are another generalization of graphs, introduced by Spielberg \cite{spielberg}, which allow identifications under conditions less restrictive than the higher rank graph factorization property. Spielberg has defined $C^*$-algebras from categories of paths and considered some properties of these algebras \cite{spielberg}.

    
In Section \ref{section1} of this paper, we give some basic definitions. In Section \ref{section2}, we outline the graph $C^*$-algebra and free semigroupoid algebra constructions and discuss some of the properties of these operator algebras. In Section \ref{section3}, we overview the higher rank graphs of Kumjian and Pask and the analogous operator algebras associated to them. Finally, in Section \ref{section4}, we discuss Spielberg's $C^*$-algebra construction and results for categories of paths, and we define and propose some results for free semigroupoid algebras from categories of paths. The main goal of this survey article is to bring all work related to operator algebras associated with graphs and categories of paths under one roof, as well as introduce free semigroupoid algebras from categories of paths. Throughout the paper, we give examples of specific graphs, higher rank graphs, and categories of paths, along with the operator algebras they generate.

\section{Definitions}\label{section1}

Let $G=(G^0, G^1, r, s)$ be a directed graph consisting of countable sets $G^0$ and $G^1$ and functions $r, s: G^1\rightarrow G^0$. The elements of $G^0$ are called \emph{vertices}, and the elements of $G^1$ are called \emph{edges}. For each edge $e\in G^1$, $s(e)$ and $r(e)$ are called the \emph{source} and \emph{range} of $e$, respectively. A \emph{source} is a vertex that receives no edges. 

A \emph{path of length $n$} in $G$ is a sequence of edges $\mu = \mu_1\mu_2\dots\mu_n$ such that $s(\mu_i)=r(\mu_{i+1})$ for all $1\leq i\leq n-1.$ We write $|\mu| = n$. Notice that we concatenate paths from right to left, to be consistent with composition of the operators that will be associated with these edges. Likewise, we define the range and source of a path $\mu = \mu_1\mu_2\dots\mu_n$ by $r(\mu)=r(\mu_1)$ and $s(\mu)=s(\mu_n)$ for $|\mu|>1$, and $r(v) = v = s(v)$ for all $v\in G^0$.

We often represent graphs with diagrams where either a dot or the name of the vertex represents each vertex, and an arrow (possibly with a label) represents each edge. For example, the graph with two vertices $v_1$ and $v_2$, and two edges $e$ with $s(e) = r(e) = v_1$ and $f$ with $s(f) = v_1$ and $r(f) = v_2$, can be represented as:
\begin{center}
\begin{tikzpicture}[->,>=stealth',shorten >=1pt,auto,node distance=3cm,
                    thick,main node/.style={circle}]

  \node[main node] (1) {$v_1$};
\node[] at (2,0) (2) {$v_2$};

  \path[->,draw,thick]
        (1) edge [loop left] node {$e$} (1);

    \path[->,draw,thick]
    (1) edge node {$f$} (2);

\end{tikzpicture}
\end{center}

We will not generally require graphs to be finite; for example, the graph represented by

\begin{center}
\begin{tikzpicture}[->,>=stealth',shorten >=1pt,auto,node distance=3cm,
                    thick,main node/.style={circle}]

  \node[main node] (1) {$v_1$};
\node[] at (2,0) (2) {$v_2$};
\node[] at (4,0) (4) {\dots};
\node[] at (6,0) (3) {$v_n$};
\node[] at (8,0) (5) {\dots};

  \path[->,draw,thick]
        (1) edge node {} (2);
  
    \path[->,draw,thick]
        (2) edge node {} (4);
        \path[->,draw,thick]
        (4) edge node {} (3);
        \path[->,draw,thick]
        (3) edge node {} (5);

\end{tikzpicture}
\end{center}
has an infinite number of both vertices and edges. However, we will often be interested in graphs where each vertex receives at most finitely many edges; that is, $\{ e \in G^1 : r(e) = v\}$ is finite for all $v \in G^0$. Such a graph is called \emph{row-finite}.

In general, the process for defining an operator algebra from a graph starts with assigning projections to each vertex and partial isometries to each edge. Let $\mathcal{H}$ be a Hilbert space and $\mathcal B(\mathcal H)$ the space of bounded linear operators on $\mathcal H$.

\begin{definition}\label{defn2.1}
Let $G$ be a row-finite directed graph. A \emph{Cuntz-Krieger $G$-family} in $\mathcal{H}$ is a family $\{S, P\}$ of partial isometries $\{S_e : e \in G^1\}$ and pairwise orthogonal projections $\{P_v : v \in G^0\}$ on $\mathcal{H}$ satisfying
\begin{enumerate}
    \item[] (I) $S_e^*S_e = P_{s(e)}$ for all $e \in G^1$
    \item[] (CK) $P_v = \sum\limits_{e \in G^1: r(e) = v} S_eS_e^*$ for each $v \in G^0$ that is not a source.
\end{enumerate}
Note that the sum in (CK) is finite because the graph is assumed to be row-finite. 

A \emph{Cuntz-Krieger-Toeplitz (CKT) $G$-family} satisfies (I) as well as
\begin{enumerate}
    \item[] (CKT) $P_v < \sum\limits_{e \in G^1: r(e) = v} S_eS_e^*$ for each $v \in G^0$ that is not a source.
\end{enumerate}
\end{definition}

Conditions (I) and (CK) are called the \emph{Cuntz-Krieger relations}, and Conditions (I) and (CKT) are called the \emph{Cuntz-Krieger-Toeplitz relations}. Note that Condition (I) implies that $S_e$ is a partial isometry with initial space $P_{s(e)}\mathcal{H}$. Additionally, (CK) and (CKT) imply that the partial isometries $S_e$ associated to the edges $e$ with $r(e) = v$ have mutually orthogonal ranges which are closed subspaces of $P_{r(e)}\mathcal{H}$.

In Section \ref{section2}, we will use these families of isometries to define the graph $C^*$-algebra and the free semigroupoid algebra of a graph. First, however, we define some properties of operator algebras that we will explore for the algebras discussed in this survey.

An operator $T \in \mathcal{B}(\mathcal{H})$ is called \emph{nilpotent} if $T^n = 0$. We say that $T$ is \emph{quasinilpotent} if the spectrum of $T$ is 0, or, equivalently, if $\lim\limits_{n \to \infty} \|T^n\|^{1/n} = 0$. The \emph{Jacobson radical} $\text{rad}(\mathcal{A})$ of a Banach algebra $\mathcal{A}$ is the intersection of the kernels of all algebraically irreducible representations. An algebra $\mathcal{A}$ is called \emph{semisimple} if $\text{rad}(\mathcal{A}) = 0$. It is a well-known fact (for example, Theorem 2.3.5(ii) in \cite{rickart}), that the Jacobson radical of an algebra of operators is the largest quasinilpotent ideal in the algebra.

The \emph{commutant} $\mathcal{A}'$ of an algebra $\mathcal{A}$ is the set of all operators in $\mathcal{B}(\mathcal{H})$ that commute with all operators in $\mathcal{A}$. The \emph{bicommutant} of $\mathcal{A}$ is the commutant of $\mathcal{A}'$. We will look at when the bicommutant of an algebra is equal to the algebra itself. The famous von Neumann Bicommutant Theorem states that if we have an algebra $\mathcal A$ consisting of bounded operators on a Hilbert space that contains the identity operator and is closed under taking adjoints, then the closures of $\mathcal A$ in the weak operator topology and the strong operator topology are equal, and are in turn equal to the bicommutant $\mathcal A''$ of $\mathcal A$ (see Corollary 3.3 of \cite{stratila.zsido}). Although the free semigroupoid algebras that we will be examining may not be von Neumann algebras, and will just be weak operator topology (WOT)-closed algebras, they still satisfy that the bicommutant is equal to itself, as we will see in Section \ref{section2}.

Finally, for the non-self-adjoint algebras, we will be interested in whether or not the algebra is reflexive and hyper-reflexive. Roughly speaking, a reflexive algebra is one that can be characterized by its invariant subspaces. A subspace $M$ of a Hilbert space $\mathcal{H}$ is \emph{invariant} for an operator $A \in \mathcal{B}(\mathcal{H})$ if $A(M) \subseteq M$. For a subalgebra $\mathcal{A}$ of $\mathcal{B}(\mathcal{H})$, the set of all subspaces that are invariant for all operators in $\mathcal{A}$ forms a lattice, written $\text{Lat}(\mathcal{A})$. The set of all operators in $\mathcal{B}(\mathcal{H})$ for which all subspaces in $\text{Lat}(\mathcal{A})$ are invariant forms an algebra, written $\text{Alg Lat}(\mathcal{A})$. It is immediate that $\mathcal{A} \subseteq \text{Alg Lat}(\mathcal{A})$. When the opposite containment holds, $\mathcal{A}$ is called \emph{reflexive}. The notion of reflexivity was introduced by Radjavi and Rosenthal in \cite{radjavi.rosenthal}, and terminology was suggested by Halmos in \cite{halmos, halmos1}.

All von Neumann algebras are reflexive algebras (see Theorem 9.17 of \cite{radjavi}). Since all reflexive operator algebras are weakly closed subalgebras containing the identity, it follows that von Neumann algebras are precisely the self-adjoint reflexive operator algebras. Another important class of examples of reflexive algebras are nest algebras, which in finite dimensions are algebras of upper triangular matrices. These were introduced by Ringrose in \cite{ringrose} as an example of reflexive operator algebras. For more examples of  reflexive algebras, see \cite{arias.popescu.1995, helmer, murray, peligrad1, peligrad2, hyper3}. 

Hyper-reflexivity, defined by Arveson in \cite{arveson}, is a stronger condition than reflexivity and is defined as follows. Let $\mathcal{L} = \text{Lat}(\mathcal{A})$. Then $\mathcal{L}$ determines a seminorm on $\mathcal{B}(\mathcal{H})$ by
\[ \beta_{\mathcal{L}}(T) = \sup\limits_{L \in \mathcal{L}} \|P_L^\perp T P_L\|,\]
where $P_L$ is the projection onto the subspace $L$. Then $\beta_\mathcal{L}(T) = 0$ if $T \in \text{Alg}(\mathcal{L})$. Thus, $\mathcal A$ is reflexive if and only if $$\mathcal A = \{T\in\mathcal B(\mathcal H): \beta_{\mathcal{L}}(T) =0\}.$$ Note that this is equivalent to : For every $T\in\mathcal B(\mathcal H)$, there exists a constant $C_T$ such that $\text{dist}(T, \text{Alg}(\mathcal{L}))\leq C_T \beta_{\mathcal L}(T)$. Moreover, we have
\[ \beta_\mathcal{L}(T) \leq \text{dist}(T, \text{Alg}(\mathcal{L}))\]
for all $T \in \mathcal{B}(\mathcal{H})$. The algebra is said to be \emph{hyper-reflexive} if these norms are comparable, and the constant of reflexivity is the smallest value of $C$ such that 
\[ \text{dist}(T, \text{Alg}(\mathcal{L})) \leq C \beta_{\mathcal L}(T)\]
Clearly, all hyper-reflexive operator algebras are reflexive operator algebras. But converse may not be true, see \cite{ reflexivecounter2, reflexivecounter1}. Nest algebras are also hyper-reflexive, with distance constant one (see Theorem 1.1 of \cite{arveson}). For more about nest algebras and hyper-reflexivity, we refer readers to \cite{tenlectures, nestalgebras, distanceformulae}. For more examples of hyper-reflexive algebras, see \cite{hyper2, davidson.1987, hyper6, hyper5, hyper1}. 


\section{Graph $C^*$-algebras and Free Semigroupoid Algebras}\label{section2}

We begin this section by looking at $C^*$-algebras generated from graphs. The $C^*$-algebra generated by a Cuntz-Krieger family $\{S, P\}$ is written $C^*(S,P)$. The \emph{graph $C^*$-algebra} of graph $G$, written as $C^*(G)$, is defined as the universal $C^*$-algebra generated by families of isometries satisfying the Cuntz-Krieger relations. The graph $C^*$-algebra $C^*(G)$ always exists and is unique up to isomorphism sending generators to generators (see Proposition 1.21 and Corollary 1.22 of \cite{raeburn}). A good introduction to graph $C^*$-algebras is Raeburn's book \cite{raeburn}. 

We mention here an important result called the Cuntz-Krieger Uniqueness Theorem that is useful in finding graph $C^*$-algebras. A \emph{cycle} is a path $\mu = \mu_1 \mu_2 \dots \mu_n$ with $n \geq 1$, $s(\mu_n) = r(\mu_1)$, and $s(\mu_i)\neq s(\mu_j)$ for all $i\neq j$. An edge $e$ is an \emph{entry} to the cycle $\mu$ if there exists $i$ such that $r(e) = r(\mu_i)$ and $e\neq \mu_i$.

\begin{theorem}[\cite{rowfinite}, Theorem 3.1] Suppose $G$ is a row-finite directed graph in which every cycle has an entry, and $\{T,Q\}$ is a Cuntz-Krieger $G$-family in a Hilbert space $\mathcal H$ such that $Q_v\neq 0$ for every $v\in G^0$. Then $C^*(T,Q)$ is isomorphic to $C^*(G)$.
\end{theorem}

As stated in Remark 2.17 of \cite{raeburn}, the above theorem for finite graphs follows as a special case of Theorem 2.13 of \cite{cuntz}. Thus, the above theorem is essentially due to Cuntz and Krieger. The condition `every cycle has an entry' was introduced by Kumjian, Pask, and Raeburn in \cite{entry} and the version stated above was first proved as Theorem 3.1 of \cite{rowfinite}. Another important uniqueness result is the Gauge Invariant Uniqueness Theorem (see Theorem 2.1 of \cite{rowfinite}).

We now give a few examples of directed graphs and the graph $C^*$-algebras they generate. 

\begin{example}\label{singleloop} (\cite{raeburn}, Example 1.23) Let $G$ be the graph with one vertex $x$ and one edge $e$. A Cuntz-Krieger $G$-family $\{S, P\}$ on $\mathcal H$ must satisfy $S_e^*S_e = P_v = S_eS_e^*$. Thus, any Cuntz-Krieger $G$-family consists of a unitary and the identity. So $C^*(G)$ is the universal $C^*$-algebra generated by a unitary, and so is isomorphic to $C(\mathbb{T})$, the complex-valued continuous functions on the unit circle $\mathbb{T}$ in $\mathbb R^2$.
\end{example}

\begin{example} If $G$ is the graph with one vertex and $n \geq 2$ edges, then any Cuntz-Krieger $G$-family $\{P,S\}$ satisfies $P = I$ and $\sum\limits_{i=1}^n S_i = I$, and so $C^*(G)$ is the Cuntz algebra $\mathcal{O}_n$.  For a detailed study of $\mathcal O_n$, see \cite{On}.
\end{example}

\begin{example} Consider the graph $G$ with $n$ vertices $x_1, \dots, x_n$ and $n-1$ edges $e_1, \dots, e_{n-1}$ satisfying $s(e_j) = x_j$ and $r(e_j) = x_{j+1}$:

\begin{center}
\begin{tikzpicture}[->,>=stealth',shorten >=1pt,auto,node distance=3cm,
                    thick,main node/.style={circle}]

  \node[main node] (1) {$x_1$};
\node[] at (2,0) (2) {$x_2$};
\node[] at (4,0) (4) {\dots};
\node[] at (6,0) (3) {$x_n$};

  \path[->,draw,thick]
        (1) edge node {$e_1$} (2);
  
    \path[->,draw,thick]
        (2) edge node {$e_2$} (4);
        \path[->,draw,thick]
        (4) edge node {$e_{n-1}$} (3);
       
\end{tikzpicture}
\end{center}
Let $H = \mathbb{C}^n$ with basis $\{b_1, \dots b_n$\}. Let $P_{e_j}$ be the matrix unit $E_{jj}$, and let $S_{e_j}$ be the matrix unit $E_{j,j+1}$
Then $\{S, P\}$ is a Cuntz-Krieger $G$-family with $C^*(S,P) = M_n(\mathbb{C})$. Since $G$ has no cycles, the Cuntz-Krieger Uniqueness Theorem applies, giving us that $C^*(G)$ is isomorphic to $M_n(\mathbb{C})$.
\end{example}
    
\begin{example}\label{compacts} Consider the graph $G$ with infinite vertices $\{x_n\}_{n\in \mathbb{N}}$ and infinite edges $e_n$ satisfying $s(e_n) = x_n$ and $r(e_n) = x_{n+1}$:
        \begin{center}
\begin{tikzpicture}[->,>=stealth',shorten >=1pt,auto,node distance=3cm,
                    thick,main node/.style={circle}]

  \node[main node] (1) {$x_1$};
\node[] at (2,0) (2) {$x_2$};
\node[] at (4,0) (4) {\dots};
\node[] at (6,0) (3) {$x_n$};
\node[] at (8,0) (5) {\dots};

  \path[->,draw,thick]
        (1) edge node {$e_1$} (2);
  
    \path[->,draw,thick]
        (2) edge node {$e_2$} (4);
        \path[->,draw,thick]
        (4) edge node {$e_{n-1}$} (3);
       \path[->,draw,thick]
        (3) edge node {$e_{n}$} (5);

\end{tikzpicture}
\end{center}

Let $\mathcal{H} = \ell^2$ with basis $\{b_n\}_{n \in \mathbb{N}}$. Let $P_{x_n}$ be the projection onto span$\{b_n\}$ and let $S_{e_n}$ be given by
\[ S_{e_n}(b_m) = \left\{ \begin{array}{cc} b_{m+1} & \text{if $m=n$} \\ 0 & \text{else} \end{array} \right. .\]
Then $\{S, P\}$ is a Cuntz-Krieger $G$-family with $C^*(S,P)$ equal to the compact operators on separable Hilbert space $\mathcal{K}(\mathcal{H})$. Since $G$ has no cycles, the Cuntz-Krieger Uniqueness Theorem applies, giving us that $C^*(G)$ is isomorphic to $\mathcal{K}(\mathcal{H})$.

Note that a similar argument shows that the graph
        \begin{center}
\begin{tikzpicture}[->,>=stealth',shorten >=1pt,auto,node distance=3cm,
                    thick,main node/.style={circle}]

  \node[main node] (1) {$x_1$};
\node[] at (2,0) (2) {$x_2$};
\node[] at (4,0) (4) {\dots};
\node[] at (6,0) (3) {$x_n$};
\node[] at (8,0) (5) {\dots};

  \path[->,draw,thick]
        (2) edge node {$e_1$} (1);
  
    \path[->,draw,thick]
        (4) edge node {$e_2$} (2);
        \path[->,draw,thick]
        (3) edge node {$e_{n-1}$} (4);
       \path[->,draw,thick]
        (5) edge node {$e_{n}$} (3);

\end{tikzpicture}
\end{center}
also has graph $C^*$-algebra isomorphic to $\mathcal{K}(\mathcal{H})$.
\end{example}

\begin{example}\label{toeplitz} Consider the graph $G$ with two vertices $v_1$ and $v_2$, and two edges, $e$ with $s(e) = r(e) = v_1$ and $f$ with $s(f) = v_1$ and $r(f) = v_2$:
    \begin{center}
\begin{tikzpicture}[->,>=stealth',shorten >=1pt,auto,node distance=3cm,
                    thick,main node/.style={circle}]

  \node[main node] (1) {$v_1$};
\node[] at (2,0) (2) {$v_2$};

  \path[->,draw,thick]
        (1) edge [loop left] node {$e$} (1);

    \path[->,draw,thick]
    (1) edge node {$f$} (2);

\end{tikzpicture}
\end{center}
A Cuntz-Krieger $G$ family $\{S, P\}$ must satisfy $S_e^*S_e = P_{v_1} = S_f^*S_f$, $P_{v_1} = S_eS_e^*$, and $P_{v_2} = S_fS_f^*$. Note that $(S_e + S_f)(S_e + S_f)^* = I$, $(S_e + S_f)^*(S_e + S_f) = 2P_{v_1}$ and $I - P_{v_1} = P_{v_2}$. So any Cuntz-Krieger $G$-family is generated by $S_e + S_f$, which is a non-unitary isometry when $P_{v_2} \neq 0$. Thus, $C^*(G)$ is unitarily equivalent to the Toeplitz algebra.
\end{example}

\begin{example}
Let $C_n$ be the graph with $n$ vertices $v_1, \dots, v_n$ and $n$ edges $e_1, \dots, e_n$ satisfying $r(e_j) = v_{j+1} = s(e_{j+1})$ for $j<n$ and $r(e_n) = v_1 = s(e_1)$. The graph $C_4$ is pictured below: 
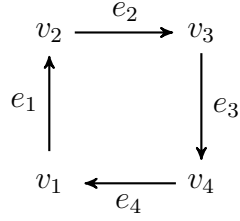
\begin{center}
\begin{tikzpicture}[->,>=stealth',shorten >=1pt,auto,node distance=3cm,
                    thick,main node/.style={circle}]

  \node[main node] (00) {$v_1$};
  \node[] at (0,2) (02) {$v_2$};
  \node[] at (2,2) (22) {$v_3$};
\node[] at (2,0) (20) {$v_4$};

    \path[->,draw,thick]
        (00) edge node {$e_1$} (02);
    \path[->,draw,thick]
        (02) edge node {$e_2$} (22);
    \path[->,draw,thick]
        (22) edge node {$e_3$} (20);
    \path[->,draw,thick]
        (20) edge node {$e_4$} (00);

\end{tikzpicture}
\end{center}

Then $C^*(C_n)$ is isomorphic to $C(\mathbb{T}, M_n(\mathbb{C}))$. For more details, see Lemma 2.4 of \cite{huef.raeburn} and Theorem 2.2 of \cite{evans}.
\end{example}

In addition to graph $C^*$-algebras, there are also non-self-adjoint operator algebras associated to graphs. Given a graph $G$, a free semigroupoid algebra is the unital WOT-closed algebra generated by a CKT $G$-family. The most common such algebra is the \emph{left regular free semigroupoid algebra}, which is generated by the left regular representation on the \emph{path space} $\mathbb{F}^+(G)$ of $G$, which is the set of all paths in $G$. More specifically, the left regular free semigroupoid algebra of a graph $G$ is defined as follows. Let $\mathcal{H}_G$ be a Hilbert space with orthonormal basis $\{\xi_{w}\}_{w \in \mathbb{F}^+(G)}$, indexed by the path space of $G$. This is called a Fock space. For each $w \in \mathbb{F}^+(G)$, we can define a linear operator $L_w \in \mathcal{B}(\mathcal{H}_G)$ as follows. For $\mu \in \mathbb{F}^+(G)$, let
\[ L_w(\xi_\mu) = \left\{ \begin{array}{cc} \xi_{w\mu} & \text{if $s(w) = r(\mu)$} \\ 0 & \text{else} \end{array} \right. .\]
Then $L_w$ is a partial isometry on the Fock space, sometimes called a \emph{partial creation operator}.

Notice that if $e$ and $f$ are edges, then $L_e$ and $L_f$ have orthogonal ranges. Also, for any vertex $x \in G^0$, $L_x$ is a projection:
\[ L_x \xi_\nu = \left\{ \begin{array}{cc} \xi_{\nu} & \text{if $r(\nu) = x$} \\ 0 & \text{else} \end{array} \right. .\] 
For a path $\mu = e_1 e_2 \dots e_n$, let $L_\mu = L_{e_1} L_{e_2} \dots L_{e_n}$. Then $\{L_\mu\}_{\mu \in \mathbb{F}^+(G)}$ is a CKT $G$-family as defined in Section \ref{section1}.

\begin{definition}[\cite{kribs.power.2004}, Definition 3.2]
Let $\mathfrak{L}_G$ be the WOT-closed algebra generated by $\{L_w\}_{w \in \mathbb{F}^+(G)}$. This is called the \emph{(left regular) free semigroupoid algebra}. 
\end{definition}

\begin{remark}
The norm-closed algebra generated by $\{L_\mu\}_{\mu \in \mathbb{F}^+(G)}$ for a finite graph $G$ is called a \emph{quiver algebra}. For more on quiver algebras, see \cite{katsoulis.kribs.2004, muhly.1997, muhly.solel.1998, muhly.solel.1999}.
\end{remark}

We now include a few examples of left regular free semigroupoid algebras generated by graphs.

\begin{example}[\cite{kribs.power.2004}, Example 6.1]
Consider the graph $G$ with a single vertex $x$ and a single edge $e$, as in Example \ref{singleloop}. The Hilbert space $\mathcal{H}_G$ is isomorphic to the Hardy space $H^2$, and $L_e$ and $L_x$ are isomorphic to the unilateral shift and the identity operator, respectively. Thus, $\mathfrak{L}_G$ is isomorphic to the WOT-closed algebra generated by those two operators, which is $H^\infty$.
\end{example}

\begin{example} If $G$ has only a single vertex, then the unital WOT-closed algebra generated by a CKT family is called a \emph{free semigroup algebra}, and can also be said to be generated by a family of $n$ isometries with orthogonal ranges. The left regular free semigroup algebra, also called the non-commutative analytic Toeplitz algebra, is written $\mathfrak{L}_n$. It was introduced by Popescu \cite{popescu1, popescu2, popescu3} and has been studied extensively by Arias and Popescu \cite{arias.popescu.1995}, and Davidson and Pitts \cite{davidson.pitts.1999}.

Free semigroup algebras other than the left regular free semigroup algebra have also been studied. Davidson, Katsoulis, and Pitts prove a structure theorem for all free semigroup algebras in \cite{davidson.katsoulis.pitts}. Read gives an example of a CKT family for which the free semigroup algebra is $\mathcal{B}(\mathcal{H})$ (\cite{read}; see also \cite{davidson.2006}). Davidson gives other examples in \cite{davidson.survey}. For more work on free semigroup algebras, see \cite{davidson.pitts.shpigel, davidson.li.pitts.2001, davidson.pitts.1998, davidson.pitts.1998.2, davidson.wright, kribs.2001}.
\end{example}

\begin{example}[\cite{kribs.power.2004}, Example 6.3] Consider the graph $G$ given in Example \ref{toeplitz}. Then $\mathfrak{L}_G$ is generated by $L_e, L_f, L_x$, and $L_y$. If we make the identifications $\mathcal{H}_G = L_x \mathcal{H}_G \oplus L_y \mathcal{H}_G \cong H^2 \oplus H^2$, where $H^2$ is the Hardy space, then 
\[ L_e \cong \left[ \begin{array}{cc} S & 0 \\ 0 & 0 \end{array} \right]; \ L_f \cong \left[ \begin{array}{cc} 0 & 0 \\ S & 0 \end{array} \right]; \ L_e \cong \left[ \begin{array}{cc} I & 0 \\ 0 & 0 \end{array} \right]; \ L_e \cong \left[ \begin{array}{cc} 0 & 0 \\ 0 & I \end{array} \right],\]
where $S$ is the unilateral shift. Thus, $\mathfrak{L}_G$ is unitarily equivalent to
\[ \mathfrak{L}_G \cong \left[ \begin{array}{cc} H^\infty & 0 \\ H^\infty_0 & \mathbb{C}I \end{array} \right].\]
\end{example}

We now consider some properties of free semigroupoid algebras. Analogous to the left regular free semigroupoid algebra, we can use the right regular representation on $\mathbb{F}^+(G)$ to define a right analogue to $\mathfrak{L}_G$. Given $\mu \in \mathbb{F}^+(G)$, we define an operator $R_{\tilde{\mu}}$ by
\[ R_{\tilde{\mu}} \xi_\nu = \left\{ \begin{array}{cc} \xi_{\nu\mu} & \text{if $r(\mu) = s(\nu)$} \\ 0 & \text{else} \end{array} \right.,\]
where $\tilde{\mu}$ is $\mu$ with the edges in the reverse order.
Let $\mathfrak{R}_G$ be the WOT-closed algebra generated by $\{R_{\tilde{\mu}}\}_{\mu \in \mathbb{F}^+(G)}$.

If $A \in \mathfrak{L}_G$, then clearly $A \in \mathfrak{R}_G'$, since for any $\mu, \nu, w \in \mathbb{F}^+(G)$, \[ L_\nu R_{\tilde{\mu}} \xi_w = L_\nu \xi_{w \mu} = \xi_{\nu w \mu} = R_{\tilde{\mu}} \xi_{\nu w} = R_\mu L_\nu \xi_x.\]

The equality $\mathfrak{L}_G = \mathfrak{R}_G'$ was shown by Davidson and Pitts for the one vertex case in \cite{davidson.pitts.1999}, and by Kribs and Power in the general case in \cite{kribs.power.2004}. These authors also show that $\mathfrak{R}_G = \mathfrak{L}_G'$, so it follows that $\mathfrak{L}_G$ is its own double commutant, $\mathfrak{L}_G = \mathfrak{L}_G''$. The proof of these results involves a ``Fourier expansion" of $A$ which is worth outlining below. 

Let $A$ be in $\mathfrak{L}_G$. For each vertex $x \in G^0$, there are constants $\{a_w\}_{s(w) = x}$ such that
\[ A\xi_x = AL_x \xi_x = R_x(AL_x)\xi_x = \sum\limits_{s(w) = x} a_w \xi_w.\]
So for $v = xv \in \mathbb{F}^+(G)$,
\[ A\xi_v = R_vA\xi_x = \sum\limits_{s(w) = x} a_w \xi_{wv}.\]
We would like to conclude that $A$ is equal to the (possibly infinite) sum $\sum_{w \in \mathbb{F}^+(G)} a_w L_w$ in some sense. Davidson and Pitts \cite{davidson.pitts.1999} prove that the Cesàro partial sums of $\sum_{w \in \mathbb{F}^+(G)} a_w L_w$ converge SOT to $A$ by relating them to the following Cesàro partial sums for $A$:

\begin{definition}\label{cesaro sum defn} For a graph $G$, let $E_i$ be the projection onto $\text{span}\{ \xi_\mu : \mu \in \mathbb{F}^+(G), \ |\mu| = i\}$. The Cesàro sums of $A \in \mathcal{B}(\mathcal{H}_G)$ are given by
\[ \Sigma_k(A) = \sum\limits_{|j| < k} \bigg( 1- \frac{|j|}{k}\bigg) \Phi_j(A),\]
where $\Phi_j(A) = \sum\limits_{\ell \geq \max\{0, -j\}} E_\ell AE_{\ell + j}$ is the $i$th diagonal of $A$ in the matrix form associated to the partition $I = E_0 + E_1 + E_2 + \dots$. 
\end{definition}
The fact that these Cesàro sums of $A$ converge in the strong operator topology (SOT) to $A$ is a consequence of the following proposition, which is referenced but not explicitly proven in \cite{davidson.pitts.1999}. We outline a proof here.

\begin{proposition}\label{cesaro}
Let $A \in \mathcal{B}(\mathcal{H})$. Let $\{P_j\}$ be a sequence of finite rank projections such that $\sum\limits_{j=1}^\infty P_j = I$. The Cesàro sums of $A$, given by
\[ \Sigma_k(A) = \sum\limits_{|j| < k} \bigg( 1- \frac{|j|}{k}\bigg) \Phi_j(A),\]
where $\Phi_j(A) = \sum\limits_{\ell \geq \max\{0, -j\}} P_\ell AP_{\ell + j}$, converge SOT to $A$.
\end{proposition}

\begin{proof} 
Define a function $f_h : \mathbb{T} \to \mathcal{H}$ as follows. Let $n_j$ be the dimension of $P_j \mathcal{H} P_j$. For $\lambda \in \mathbb{T}$, let $U_\lambda$ be the following matrix:
\[ U_\lambda = \begin{bmatrix} 
\lambda^{-1} (I_{n_1}) & 0 & 0  & \dots \\
0 & \lambda^{-2} (I_{n_2}) & 0 & \dots  \\
0 & 0 &  \lambda^{-3} (I_{n_3}) & \dots \\
\vdots & \vdots & \vdots & \ddots \end{bmatrix},\]
where $I_{n_j}$ is the $n_j \times n_j$ identity matrix. Note $U_\lambda$ is a unitary, and for $A \in \mathcal{B}(\mathcal{H})$:
\[ U_\lambda A U_\lambda^{-1} =  \begin{bmatrix}
1 (A_{1,1}) & \lambda (A_{1,2}) & \lambda^2 (A_{1,3})  & \dots \\
\lambda^{-1} (A_{2,1}) & 1 (A_{2,2}) & \lambda (A_{2,3}) & \dots \\
\lambda^{-2} (A_{3,1}) & \lambda^{-1} (A_{3,2}) &  1 (A_{3,3}) & \dots \\
\vdots & \vdots & \vdots & \ddots  \end{bmatrix}.\]

In other words, conjugating $A$ by $U_\lambda^{-1}$ results in multiplying the $(i,j)$th block component of $A$ by $\lambda^{j-i}$.

Now let $A \in \mathcal{B}(\mathcal{H})$. Fix $h \in \mathcal{H}$. Define the function $f_h : \mathbb{T} \to \mathcal{H}$ by $f_h(\lambda) = U_{\lambda} A U_{\lambda^{-1}}h$. It can be shown that $f_h : \mathbb{T} \to \mathcal{H}$ is a continuous function, which implies that $g_h : [0,2\pi] \to \mathcal{H}$ given by $g_h(t) = f_h(e^{it})$ is also continuous. Furthermore,
\[ g_h(t) = \sum\limits_{j=-\infty}^\infty e^{ijt} \Phi_j(A)h.\]
It can then be shown that the Fourier coefficients of $g_h(t)$ are given by
\[z_n = \frac{1}{2\pi} \int_0^{2\pi} e^{-ins} g_h(s) \text{ d}s = \Phi_n(A)h.\]
This means that the Cesàro partial sums $\sigma_k g_h(t)$ of $g_h(t)$ are given by:
\[ \sigma_k g_h (t) =  \sum\limits_{|j| < k} e^{ijt} \bigg(\frac{k - |j|}{k}\bigg) \Phi_j(A)h =  U_{e^{it}}(\Sigma_k A)U_{e^{it}}^{-1}h.\]

Here we can apply a vector-valued version of Fejer's theorem (stated as Theorem 1.1 in \cite{cl}, for example) to conclude that for $t \in (0, 2\pi)$,  $\|\sigma_k g_h(t) - g_h(t)\| \rightarrow 0$ as $k \rightarrow \infty$. That is, $\|U_{e^{it}}\Sigma_k A U_{e^{-it}}h - U_{e^{it}} AU_{e^{-it}}h\|$ converges to 0. This holds for all $h \in \mathcal{H}$. So for $t \in (0, 2\pi)$, we have that $U_{e^{it}}\Sigma_k A U_{e^{-it}}$ converges SOT to $U_{e^{it}}A U_{e^{-it}}$. Since SOT convergence is preserved by unitary equivalence, this means that $\Sigma_k A$ converges SOT to $A$ as desired.
\end{proof}

It is a well-known fact that all $C^*$-algebras are semisimple (see II.1.6.4 of \cite{blackadar}) and therefore all graph $C^*$-algebras are semisimple. Davidson and Pitts showed that $\mathfrak{L}_n$ contains no non-trivial idempotents or non-zero quasinilpotent elements, implying that $\mathfrak{L}_n$ is also semisimple (\cite{davidson.pitts.1999}, Corollary 1.9; see Remark 1.10 in \cite{davidson.pitts.1999} for a possible earlier reference). For an arbitrary graph $G$, Kribs and Power showed that $\mathfrak{L}_G$ is semisimple if and only if every path in $\mathbb{F}^+(G)$ lies on a cycle (\cite{kribs.power.2004}, Theorem 5.1).

In Theorem 4.1 of \cite{arias.popescu.1995}, Arias and Popescu prove that $\mathfrak{L}_n$ is a reflexive algebra. Kribs and Power showed that for any graph $G$, $\mathfrak{L}_G$ is reflexive \cite{kribs.power.2004}. For the more generalized free semigroupoid algebras beyond that induced by the left regular representation, Kennedy showed that any free semigroup algebra is reflexive \cite{kennedy} and Davidson, Dor-on, and Li showed that any free semigroupoid algebra is reflexive \cite{davidson.dor-on.li}.

Concerning the stronger condition of hyper-reflexivity, Davidson \cite{davidson.1987} showed that $\mathfrak{L}_1$ is hyper-reflexive with constant at most 19. In \cite{davidson.pitts.1999}, Davidson and Pitts go further and prove that $\mathfrak{L}_n$ is hyper-reflexive for $n \geq 2$, with constant at most 51. Davidson notes in \cite{davidson.survey} that this bound can be lowered to 3 by applying Bercovici's hyper-reflexivity theorem (Theorem 3.1, \cite{bercovici}). Kennedy showed that a certain class of free semigroup algebras is hyper-reflexive \cite{kennedy}, and Jaëck and Power showed that if $G$ is any finite graph, then $\mathfrak{L}_G$ is hyper-reflexive \cite{jaeck.power}. 

Kribs and Power also showed that unitarily equivalent free semigroupoid algebras have isomorphic directed graphs (Theorem 9.1, \cite{kribs.power.2004}). This interesting fact is not true for graph $C^*$-algebras; see Example \ref{compacts}. For more discussion of isomorphic graphs and associated algebras, see \cite{katsoulis.kribs.2004}. For more on the structure of free semigroupoid algebras, see \cite{jury.kribs}. For a study of weighted versions of free semigroupoid algebras, see \cite{kribs.2004} and \cite{kribs.levene.power}.

\section{Operator Algebras from Higher rank graphs}\label{section3}

Higher rank graphs were introduced by Kumjian and Pask in \cite{kumjian.pask}, and are defined using the language of categories.
\begin{definition} A $k$-graph is a countable category $\Lambda$ with range and source maps $r$ and $s$ respectively, and a functor $d: \Lambda \rightarrow \mathbb N^k$ satisfying the following factorization property: for every $\lambda \in \Lambda$ and $m, n \in \mathbb{N}^k$ with $d(\lambda) = m + n$, there are unique elements $\mu, \nu \in \Lambda$ such that $\lambda = \mu \nu$ and $d(\mu) = m$ and $d(\nu) = n$. 

For a path with degree $n = (n_1, n_2, \dots, n_k) \in \mathbb{N}^k$, we say its length is $|n| = n_1 + n_2 + \dots + n_k$. Let $\Lambda^0$ be the set of paths with length 0 (the vertices of $\Lambda$) and let $\Lambda^1$ be the set of paths of length 1.
\end{definition}

\begin{example} Any directed graph $(G^0, G^1, r, s)$ gives a category where $\Lambda$ is the collection of all paths and $r, s$ are the range and source maps. This is a $1$-graph, satisfying the factorization property trivially since each path can be written as a concatenation of edges in one unique way. In fact, any $1$-graph $\Lambda$ comes from a directed graph.
\end{example}

As with graphs, we represent higher rank graphs in diagrams by using a dot or label for each element of $\Lambda^0$, and using arrows between the dots to represent elements of $\Lambda^1$. We represent the degree of each edge with a color. In the diagrams that follow, unbroken, dotted, and dashed arrows will represent three different colors.

\begin{example}\label{hr1} Consider the following $2$-graph, where the edges $b_1$ and $b_2$ have degree $(1,0)$, the edges $r_1$ and $r_2$ have degree $(0,1)$. In order for this graph to satisfy the factorization property, we must have the identification $b_2r_1 = r_2b_1$.
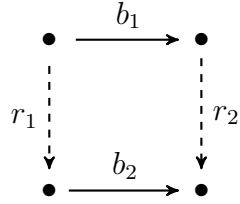
\begin{center}
\begin{tikzpicture}[->,>=stealth',shorten >=1pt,auto,node distance=3cm,
                    thick,main node/.style={circle}]

  \node[main node] (1) {$\bullet$};
\node[] at (2,0) (2) {$\bullet$};
\node[] at (0,-2) (3) {$\bullet$};
\node[] at (2,-2) (4) {$\bullet$};
 
  \path[->,draw,thick]
        (1) edge node {$b_1$} (2);
  
    \path[->,draw,thick, dashed]
        (2) edge node {$r_2$} (4);
        \path[->,draw,thick]
        (3) edge node {$b_2$} (4);
        \path[->,draw,thick, dashed]
        (1) edge node [swap]{$r_1$} (3);

\end{tikzpicture}
\end{center}
\end{example}

Such a diagram which indicates the vertices, edges, and degrees of the edges is called a 1-skeleton, and does not necessarily determine the $k$-graph, as shown in the next example.

\begin{example}(\cite{raeburn}, Example 10.5) The following $1$-skeleton corresponds to two different $2$-graphs:
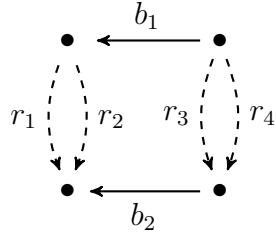
\begin{center}
\begin{tikzpicture}[->,>=stealth',shorten >=1pt,auto,node distance=3cm,
                    thick,main node/.style={circle}]

  \node[main node] (1) {$\bullet$};
\node[] at (2,0) (2) {$\bullet$};
\node[] at (0,-2) (3) {$\bullet$};
\node[] at (2,-2) (4) {$\bullet$};

    \path[->,draw,thick]
        (2) edge node [swap]{$b_1$} (1);
    \path[->,draw,thick, dashed, bend left=20]
        (2) edge node {$r_4$} (4);
    \path[->,draw,thick, dashed, bend right=20]
        (2) edge node [swap]{$r_3$} (4);
    \path[->,draw,thick]
        (4) edge node {$b_2$} (3);
    \path[->,draw,thick, dashed, bend left=20]
        (1) edge node {$r_2$} (3);
    \path[->,draw,thick, dashed, bend right=20]
        (1) edge node [swap]{$r_1$} (3);

\end{tikzpicture}
\end{center}
The path $r_1b_1$ must be identified with either $b_2r_3$ or $b_2r_4$ in order for the factorization property to be satisfied. If $r_1b_1 = b_2r_3$, then it must be true that $r_2b_1 = b_2r_4$. Likewise, if $r_1b_1 = b_2r_4$, then $r_2b_1 = b_2r_3$.
\end{example}

\begin{example}\label{loops}
Consider the $3$-graph with one edge $e$ of degree $(1,0,0)$, one edge $f$ of degree $(0,1,0)$ and one edge $g$ of degree $(0,0,1)$.

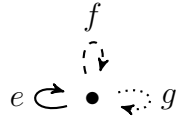
\begin{center}
\begin{tikzpicture}[->,>=stealth',shorten >=1pt,auto,node distance=3cm,
                    thick,main node/.style={circle}]

  \node[main node] (1) {$\bullet$};

  \path[->,draw,thick]
        (1) edge [loop left] node {$e$} (1);
    \path[->,draw,thick, dashed]
        (1) edge [loop above] node {$f$} (1);
    \path[->,draw,thick, dotted]
        (1) edge [loop right] node {$g$} (1);
\end{tikzpicture}
\end{center}

The factorization property implies that for each $(n, m, k) \in \mathbb{N}^3$, there is one unique path of degree $(n,m,k)$. So, for example, $ffgfe = egf^3$.
\end{example}

Next, we mention some definitions and results for higher rank graph $C^*$-algebras analogous to graph $C^*$-algebras.

\begin{definition} For a $k$-graph $\Lambda$ with degree map $d$, we define $d^{-1}(n)$ to be paths of degree $n$ for all $n\in\mathbb N^k$.  We say $\Lambda$ is \emph{row-finite} if $r^{-1}(v)\cap d^{-1}(n)$ is finite for all $v\in\Lambda^0$ and $n\in\mathbb N^k$.
\end{definition}

In Definition \ref{defn2.1} of a Cuntz-Krieger $G$-family for a row-finite directed graph $G$, we saw that the orthogonal projection $P_v$ was the sum of $S_eS_e^*$ for all $e\in G^1$ with $r(e)=v$, for each $v \in G^0$ that is not a source. So to define a Cuntz-Krieger family for higher rank graphs, we first make this as a basic assumption that motivates the following definition.

\begin{definition} A $k$-graph $\Lambda$ is said to \emph{have no source} if for every $v\in \Lambda^0$ and $n\in\mathbb N^k$, there exists $\mu \in \Lambda$ such that $d(\mu) = n$ and $r(\mu) = v$.
\end{definition}

\begin{definition} Let $\Lambda$ be a $k$-graph that is row-finite and has no source. A Cuntz-Krieger $\Lambda$-family in $\mathcal H$ is a collection of partial isometries $S=\{S_{\mu} : \mu\in\Lambda\}$ satisfying the following:
\begin{enumerate}
    \item $\{S_v: v\in\Lambda^0\}$ are mutually orthogonal projections,
    \item $S_{\mu_1} S_{\mu_2} = S_{\mu_1\mu_2}$ when $s(\mu_1) = r(\mu_2)$,
    \item $S_{\mu}^*S_{\mu} = S_{s(\mu)}$ for all $\mu\in\Lambda$,
    \item $S_v = \sum\limits_{\mu\in\Lambda, d(\mu) =n,  r(\mu) = v} S_{\mu}S_{\mu}^*$ for all $v\in \Lambda^0$ and $n\in\mathbb N^k$.
\end{enumerate}
\end{definition}
Let $\Lambda$ be a $k$-graph that is row-finite and has no sources. Let $S$ be a Cuntz-Krieger $\Lambda$-family. Then the partial isometries associated to paths of the same degree have orthogonal ranges; that is, $\{S_{\mu}S_{\mu}^* : d(\mu) = n\}$ is a mutually orthogonal family of projections for each $n\in\mathbb N^k$. It is natural to ask if there is a universal $C^*$-algebra $C^*(\Lambda)$ generated by a Cuntz-Krieger $\Lambda$-family that will have property that if $\{S_{\mu} : \mu\in\Lambda\}$ is any Cuntz-Krieger $\Lambda$-family, then there is a $*$-homomorphism from $C^*(\Lambda)$ to $C^*\left(\{S_{\mu} : \mu\in\Lambda\}\right)$. 

For any $k$-graph $\Lambda$, $C^*(\Lambda)$ always exists (see Proposition 10.9 of \cite{raeburn}) and every partial isometry in the generating family of the universal Cuntz-Krieger $\Lambda$-family is non-zero (see Proposition 2.11 of \cite{kumjian.pask} or Corollary 10.13 of \cite{raeburn}). An appropriate analogue for the Gauge Invariant Uniqueness Theorem for $k$-graphs, can be found in Theorem 3.4 of \cite{kumjian.pask}. The most important application of the Gauge Invariant Uniqueness Theorem for $k$-graphs is that for a $k$-graph $\Lambda$, there is an isomorphism between $C^*(\Lambda)$ and the $C^*$-algebra of the locally compact groupoid with unit space equal to the infinite path space of $\Lambda$. For a detailed study of groupoids and $C^*$-algebras generated by them, see \cite{groupoid}. For the details of the proof of isomorphism in the case of directed graphs, see \cite{graphgroupoid}, and for $k$-graphs, see Corollary 3.5 (i) of \cite{kumjian.pask}. 

It would be interesting to find appropriate version(s) of the Cuntz-Krieger Uniqueness Theorem for $k$-graphs. For directed graphs, in the statement of the Cuntz-Krieger Uniqueness Theorem, we had the condition that every cycle has an entry. We have a version of the Cuntz-Krieger Uniqueness Theorem for $k$-graphs under an \emph{aperiodicity condition}; see Section 4 of \cite{kumjian.pask}. 

Below we give a few examples of $C^*$-algebras generated by higher rank graphs.
\begin{example}
If $\Lambda$ is a $1$-graph associated to a row-finite directed graph $G$ with no sources, then by restricting the Cuntz-Krieger $\Lambda$-family to $G^0$ and $G^1$, we obtain a Cuntz-Krieger $G$-family in the sense of Definition \ref{defn2.1}. Thus for a directed graph $G$, we have $C^*(\Lambda)= C^*(G)$. 
\end{example}

\begin{example} Let $T$ be the semigroup $\mathbb N^k$ viewed as a small category and $d: T\rightarrow \mathbb N^k$ be the identity map. Then $(T,d)$ is a $k$-graph. We have, $C^*(T) = C(\mathbb T^k)$, where $\mathbb T^k$ denote the $k$-torus (see Example 1.7 (iii) of \cite{kumjian.pask}).
\end{example}

\begin{example} For $k\geq 1$, let $\Omega_k$ be the small category with $\Omega_k^0=\mathbb N^k$ and $$\Omega_j^*=\{(p, q)\in\mathbb N^k\times\mathbb N^k : p\leq q\text{ if and only if } p_i\leq q_i\text{ for all } 1\leq i\leq k\}.$$ Let the source and range maps be given by $r(m, n) = m$ and $s(m, n) = n$ with $d: \Omega_k\rightarrow \mathbb N^k$ defined as $d(m, n) = n-m$. Then $\Omega_k$ is a $k$-graph with $C^*(\Omega_k) = \mathcal K(l^2(\mathbb N^k))$ (see Example 1.7 (ii) of \cite{kumjian.pask}).

Below we give the $1$-skeleton for the infinite graph $\Omega_2$. 
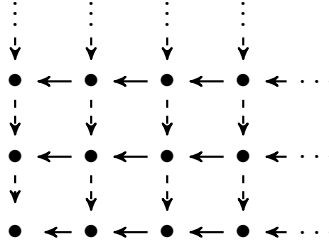
\begin{center}
\begin{tikzpicture}[->,>=stealth',shorten >=1pt,auto,node distance=3cm,
                    thick,main node/.style={circle}]

  \node[main node] (00) {$\bullet$};
  \node[] at (0,1) (01) {$\bullet$};
\node[] at (0,2) (02) {$\bullet$};
\node[] at (0,3) (03) {$\vdots$};
\node[] at (1,0) (10) {$\bullet$};
\node[] at (1,1) (11) {$\bullet$};
\node[] at (1,2) (12) {$\bullet$};
\node[] at (1,3) (13) {$\vdots$};
\node[] at (2,0) (20) {$\bullet$};
\node[] at (2,1) (21) {$\bullet$};
\node[] at (2,2) (22) {$\bullet$};
\node[] at (2,3) (23) {$\vdots$};
\node[] at (3,0) (30) {$\bullet$};
\node[] at (3,1) (31) {$\bullet$};
\node[] at (3,2) (32) {$\bullet$};
\node[] at (3,3) (33) {$\vdots$};
\node[] at (4,0) (40) {$\dots$};
\node[] at (4,1) (41) {$\dots$};
\node[] at (4,2) (42) {$\dots$};
 
    \path[->,draw,thick, dashed]
        (01) edge node {} (00);
    \path[->,draw,thick, dashed]
        (11) edge node {} (10);
    \path[->,draw,thick, dashed]
        (21) edge node {} (20);
    \path[->,draw,thick, dashed]
        (31) edge node {} (30);
        
    \path[->,draw,thick, dashed]
        (02) edge node {} (01);
    \path[->,draw,thick, dashed]
        (12) edge node {} (11);
    \path[->,draw,thick, dashed]
        (22) edge node {} (21);
    \path[->,draw,thick, dashed]
        (32) edge node {} (31);
        
    \path[->,draw,thick, dashed]
        (03) edge node {} (02);
    \path[->,draw,thick, dashed]
        (13) edge node {} (12);
    \path[->,draw,thick, dashed]
        (23) edge node {} (22);
    \path[->,draw,thick, dashed]
        (33) edge node {} (32);

    \path[->,draw,thick]
        (10) edge node {} (00);
    \path[->,draw,thick]
        (20) edge node {} (10);
    \path[->,draw,thick]
        (30) edge node {} (20);
    \path[->,draw,thick]
        (40) edge node {} (30);
        
    \path[->,draw,thick]
        (11) edge node {} (01);
    \path[->,draw,thick]
        (21) edge node {} (11);
    \path[->,draw,thick]
        (31) edge node {} (21);
    \path[->,draw,thick]
        (41) edge node {} (31);
        
    \path[->,draw,thick]
        (12) edge node {} (02);
    \path[->,draw,thick]
        (22) edge node {} (12);
    \path[->,draw,thick]
        (32) edge node {} (22);
    \path[->,draw,thick]
        (42) edge node {} (32);

\end{tikzpicture}
\end{center}

\end{example}

We would also like to mention that $2$-graphs have been of special interest; see \cite{twograph}. 

Before moving to free semigroupoid algebras for higher rank graphs, we note that row-finiteness and having no sources have been important conditions on a $k$-graph $\Lambda$ to define $C^*(\Lambda)$. Authors have tried to relax these conditions. The procedure for dealing with $1$-graphs (or directed graphs) with sources can be done using Lemma 2.1 of \cite{rowfinite}. However, when $k>1$, there are many different kinds of sources, and there is as yet no analogous procedure for dealing with them. For some particular classes of row-finite $k$-graphs with sources with an extra local convexity condition, see \cite{withsources}. 

But the question remains : Can we analyze $k$-graphs that are not row-finite? If a vertex emits infinitely many edges, then the Cuntz-Krieger relation \begin{equation}\label{section3:eq1}S_v = \sum\limits_{\mu\in\Lambda, d(\mu) =n,  r(\mu) = v} S_{\mu}S_{\mu}^*\end{equation} does not make sense. Fowler and Raeburn noticed that even if vertices emit infinitely many edges in a $1$-graph (or directed graph), one can work with graph algebras if we merely insist that equality in \eqref{section3:eq1} occurs when $v$ emits finitely many edges; see Corollaries 4.2 and 4.5 of \cite{notrow}. For a more detailed study of this, see \cite{infinitegraphs}. For an alternative approach to studying the graph $C^*$-algebra of a directed graph, see \cite{functorial}. In \cite{finitelyaligned}, the theory of Cuntz-Krieger families and graph algebras has been generalized to the class of finitely-aligned $k$-graphs. This class contains in particular all row-finite $k$-graphs. A Gauge-Invariant
Uniqueness Theorem and a Cuntz-Krieger Uniqueness Theorem have also been proved for finitely-aligned $k$-graphs in \cite{finitelyaligned}. For the study of general $k$-graphs, the work of Spielberg in the more general setting of categories of paths, which we discuss in Section \ref{section4}, is useful.

Kribs and Power extended the free semigroupoid algebra construction to higher rank graphs in \cite{kribs.power.2006}. For a higher rank graph $\Lambda$, define a Fock space Hilbert space $\mathcal{H}_\Lambda$ with orthonormal basis $\{\xi_\mu\}_{\mu \in \Lambda}$, indexed by the elements of $\Lambda$. We can then define linear operators $L_\mu \in \mathcal{B}(\mathcal{H}_\Lambda)$ as follows. For $\nu \in \Lambda$, define:
\[ L_\mu \xi_\nu = \left\{ \begin{array}{cc} \xi_{\mu \nu} & \text{if $s(\mu) = r(\nu)$} \\ 0 & \text{else} \end{array} \right.\] 
As before, if $x$ is a vertex of $\Lambda$, then $L_x$ is a projection. 

\begin{definition}[\cite{kribs.power.2006}, Definition 3.1]
The WOT-closed algebra generated by $\{L_\mu\}_{\mu \in \Lambda}$ is called the \emph{(left regular) free semigroupoid algebra} for $\Lambda$ and is written $\mathfrak{L}_\Lambda$.
\end{definition}

The right regular free semigroupoid algebra $\mathfrak{R}_\Lambda$ can be defined analogously as in Section \ref{section2}. Kribs and Power showed that, as with free semigroupoid algebras from graphs, the commutant of $\mathfrak{L}_\Lambda$ is $\mathfrak{R}_\Lambda$, and $\mathfrak{L}_\Lambda$ is its own double commutant \cite{kribs.power.2006}.

We now include a few examples of free semigroupoid algebras from higher rank graphs.

\begin{example}
The $2$-graph in Example \ref{hr1} has 9 elements in its path space: four vertices, two edges with degree $(1,0)$, two edges with degree $(0,1)$, and a single edge of degree $(1,1)$. So $\mathcal{H}_\Lambda$ is 9-dimensional, and $\mathfrak{L}_\Lambda$ can be identified with a matrix algebra on this Hilbert space.
\end{example}

\begin{example}(\cite{kribs.power.2006}, Example 4.1)
Let $\Lambda$ be the $3$-graph with one vertex and $3$ loops $e$, $f$, and $g$, as in Example \ref{loops}. In this higher rank graph, a path $\mu$ is uniquely determined by the number of times $e$, $f$, and $g$ appear in any decomposition of $\mu$ into edges. The Hilbert space $\mathcal{H}_\Lambda$ can be identified with $\mathcal{H}_{G} \otimes \mathcal{H}_{G} \otimes \mathcal{H}_G$, where $G$ is the graph with one vertex and one edge. Then $\mathfrak{L}_\Lambda$ is unitarily equivalent to $H^\infty \otimes H^\infty \otimes H^\infty$.
\end{example}

\begin{example} (\cite{kribs.power.2006}, Example 4.3)
Define $C_n^{(k)}$ to be the $k$-graph with $n$ vertices $v_1, \dots v_n$ such that for each $i \in \{1, \dots, n-1\}$, there is one edge of each color with source $v_i$ and range $v_{i+1}$, and in addition, there is one edge of each color with source $v_n$ and range $v_1$. This is called a higher rank cyclic graph. The 1-skeleton for $C_3^{(2)}$ is:
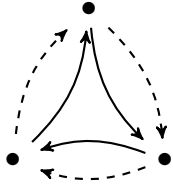
\begin{center}
\begin{tikzpicture}[->,>=stealth',shorten >=1pt,auto,node distance=3cm,
                    thick,main node/.style={circle}]

  \node[main node] (00) {$\bullet$};
  \node[] at (1,2) (02) {$\bullet$};
\node[] at (2,0) (20) {$\bullet$};
 
    \path[->,draw,thick, dashed, bend left = 20]
        (00) edge node {} (02);
    \path[->,draw,thick, dashed, bend left = 20]
        (02) edge node {} (20);
    \path[->,draw,thick, dashed, bend left = 20]
        (20) edge node {} (00);
        
    \path[->,draw,thick, bend right = 20]
        (00) edge node {} (02);
    \path[->,draw,thick, bend right = 20]
        (02) edge node {} (20);
    \path[->,draw,thick, bend right = 20]
        (20) edge node {} (00);

\end{tikzpicture}
\end{center}
Let $H^2(z,w)$ be the Hardy space for the torus $\mathbb{T}^2 = \{(z,w) : |z| = |w| = 1\}$, with its basis $\{z^p w^p : p, q \in \mathbb{N}\}$. Kribs and Power show that $\mathfrak{L}_{C_3^{(2)}}$ is unitarily equivalent to the matrix function algebra
\[\begin{bmatrix} 
H^\infty_{(3,0)}(z,w) & H^\infty_{(3,2)}(z,w) & H^\infty_{(3,1)}(z,w) \\ 
H^\infty_{(3,1)}(z,w) & H^\infty_{(3,0)}(z,w) & H^\infty_{(3,2)}(z,w) \\
H^\infty_{(3,2)}(z,w) & H^\infty_{(3,1)}(z,w) & H^\infty_{(3,0)}(z,w) \end{bmatrix}, \]
where $H^\infty_{3,i}(z,w)$ is the closed span of the basis elements $\{z^pw^q : p+q \equiv i \text{ mod } 3\}$ for $i=1,2$.
\end{example}

Kribs and Power proved that, as with free semigroupoid algebras from graphs, a higher rank free semigroupoid algebra is isomorphic to its bicommutant (\cite{kribs.power.2006}, Corollary 3.5). Also analogous to the graph case, $\mathfrak{L}_\Lambda$ is semisimple if and only if every edge in $\Lambda$ lines in a cycle (Theorem 7.2, \cite{kribs.power.2006}). In Corollary 6.3 of the same paper \cite{kribs.power.2006}, Kribs and Power apply Bercovici's hyper-reflexivity theorem to show that every single-vertex higher rank free semigroupoid algebra is reflexive. They then use this to prove reflexivity for a large class of free semigroupoid algebras. Specifically:
\begin{theorem}(\cite{kribs.power.2006}, Theorem 6.4)
Let $\Lambda$ be a higher rank graph such that no vertex $v$ satisfies all three of the following properties:
\begin{enumerate}
    \item for each $\lambda \in \Lambda$ with degree 1, $r(v) = \lambda$ implies $s(v) = \lambda$;
    \item there is at most one loop edge of each color at $v$; and
    \item there are loop edges $\mu \neq \mu'$ and paths $\lambda, \lambda' \in \Lambda$ with $s(\lambda) = v = s(\lambda')$ that immediately leave $v$ such that $\lambda \mu = \lambda' \mu'$.
\end{enumerate}
Then $\mathfrak{L}_\Lambda$ is reflexive.
\end{theorem}

In the next section, we will move on to a further generalization of graphs, and see the corresponding constructions and properties in that setting.

\section{Operator Algebras from Categories of Paths}\label{section4}

Categories of paths were introduced by Spielberg in  \cite{spielberg} as a generalization of graphs and higher rank graphs.

\begin{definition}[\cite{spielberg}, Definition 2.1]
A small category $\Lambda$ with source map $s$ and range map $r$ is called a \emph{category of paths} if, for $\alpha, \beta, \gamma \in \Lambda$,
\begin{itemize}
\item $\alpha \beta = \alpha \gamma$ implies $\beta = \gamma$ (left cancellation)
\item $\beta \alpha = \gamma \alpha$ implies $\beta = \gamma$ (right cancellation)
\item $\alpha \beta = s(\beta)$ implies $\alpha = \beta = s(\beta)$ (no inverses)
\end{itemize}
\end{definition}

\begin{example}
Graphs and higher rank graphs are examples of categories of paths. Note that the factorization property implies that left and right cancellation hold in a higher rank graph.
\end{example}

\begin{example}
The $p$-graphs introduced by Brownlowe, Sims, and Vittadello in \cite{pgraphs} are an example of a category of paths. A $p$-graph is a generalization of a directed graph where paths have a degree in a semigroup $P$ rather than a length in $\mathbb{N}$. See \cite{pgraphs} for more details.
\end{example}

\begin{example}\label{cat path}
Consider the graph
\begin{tikzcd}
    x_1 \arrow[bend right=20]{r}[swap]{b_1} \arrow[bend left=20]{r}{a_1} & x_2 \arrow[bend right=20]{r}[swap]{b_2} \arrow[bend left=20]{r}{a_2} & x_3
\end{tikzcd}
along with the identifications $b_2a_1 = a_2b_1$ and $a_2 a_1 = b_2 b_1$. This forms a category of paths with three vertices ($x_1, x_2, x_3)$ and six non-vertex paths ($a_1, b_1, a_2, b_2, a_2 a_{1}, a_{2} b_{1})$. This is not a higher rank graph, since there is no way to assign degrees to the edges in such a way that the factorization property would be satisfied.
\end{example}

\begin{example}\label{3loop}
Consider the category of paths defined by the graph
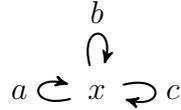
\begin{center}
\begin{tikzpicture}[->,>=stealth',shorten >=1pt,auto,node distance=3cm,
                    thick,main node/.style={circle}]

  \node[main node] (1) {$x$};

  \path[->,draw,thick]
        (1) edge [loop left] node {$a$} (1);
    \path[->,draw,thick]
        (1) edge [loop above] node {$b$} (1);
    \path[->,draw,thick]
        (1) edge [loop right] node {$c$} (1);
\end{tikzpicture}
\end{center}
along with the identifications $ab = bc = ca$, $ac = cb = ba$, and $a^2 = b^2 = c^2$. For this category of paths, we can define the length of a path in the normal sense. The identifications imply that for every $k \geq 1$, there are exactly 3 paths of length $k$. For $k \geq 2$, they can be written in the forms: $a^k, ba^{k-1}, ca^{k-1}$.
\end{example}

Our construction of $C^*(\Lambda)$ for a $k$-graph or directed graph $\Lambda$ was heavily based on the degree functor (that was used in the definition of a Cuntz-Krieger $\Lambda$-family). Indeed, for a category of paths, many different degrees can exist, giving different decompositions of the algebra (see Section 9 of \cite{spielberg}). Exel \cite{exel} proposed an idea that a degree functor is not at all necessary for defining the $C^*$-algebra.

For a semigroupoid $(\Lambda, \Lambda^{(2)}, \cdot)$ (see Definition 2.1 of \cite{exel}), a representation $f$ of $\Lambda$ in a unital $C^*$-algebra $B$ was defined in Definition 4.1 of \cite{exel}. Using the notion of tight representations (see Definition 4.5 of \cite{exel}), the universal representation $\tilde{\mathcal O}(\Lambda)$ of a semigroupoid $\Lambda$ was defined in Definition 4.6 of \cite{exel}. The most interesting thing is that any small category is a sort of semigroupoid in several ways, and it turns out that the universal representation $\tilde{\mathcal O}(\Lambda)$ of a $k$-graph, viewed as a semigroupoid, is isomorphic to $C^*(\Lambda)$ as defined in last section (see Theorem 8.7 of \cite{exel}). We saw in Section \ref{section3} that for a $k$-graph $\Lambda$, as an application of the Gauge Invariant Uniqueness Theorem for $k$-graphs, there is isomorphism between $C^*(\Lambda)$ and the $C^*$-algebra of a locally compact groupoid with unit space equal to the infinite path space of $\Lambda$. Exel \cite{exel1} also observed that any semigroupoid $C^*$-algebra can be described as a groupoid $C^*$-algebra.

The work of Exel motivated Spielberg to define the $C^*$-algebra of a category of paths. The idea of Spielberg's construction was based on an analogy from Cuntz and Krieger's works: taking ultrafilters on a Boolean ring generated by a set of paths extending a particular path and defining the $C^*$-algebras via a groupoid (this produces the usual Cuntz-Krieger algebra). Since this process just uses the operations of concatenation and cancellation, this construction can be translated to categories of paths. Spielberg not only found a generalization of $C^*(\Lambda)$ for a $k$-graph $\Lambda$ to categories of paths, he also found generalization of gauge actions and aperiodicity for categories of paths in Section 9 and Section 10 of \cite{spielberg}. We will refer readers to \cite{spielberg} for the technical details and we instead discuss here the extension of free semigroupoid algebras to categories of paths and our proposal to extend results of free semigroupoid algebras associated with $k$-graphs to those of categories of paths. The following material was done in PhD work of the first author and can be accessed from \cite{thesis} and is under preparation in paper form in \cite{unpublished}.

The same construction used for the left regular free semigroupoid algebra from a graph can be easily applied to a category of paths. For a category of paths $\Lambda$, define a Hilbert space $\mathcal{H}_\Lambda$ with basis $\{\xi_\mu\}_{\mu \in \Lambda}$ indexed by the path space of the category of paths, and for each $\mu \in \Lambda$, define partial isometries $L_\mu$ by 
\[ L_\mu(\xi_\nu) = \left\{ \begin{array}{cc} \xi_{\mu\nu} & \text{if $s(\mu) = r(\nu)$} \\ 0 & \text{else} \end{array} \right.,\]
for each $\nu \in \Lambda$.

\begin{definition}
The \emph{left regular free semigroupoid algebra} $\mathfrak{L}_\Lambda$ is the WOT-closed algebra generated by the operators $\{L_\mu\}_{\mu \in \Lambda}$.
\end{definition}

\begin{example}\label{3loop2}(see \cite{thesis}, Lemma 3.4.4)
Let $\Lambda$ be the category of paths from Example \ref{3loop}. As mentioned, this category of paths has exactly 3 paths of length $k$ for every $k \geq 1$. For $k \geq 2$, they can be written in the forms: $a^k, ba^{k-1}, ca^{k-1}$. Consider the Hilbert spaces $\{H_k\}_{k \geq 0}$ where $H_0 = \{x\}$, $H_1 = \{a, b, c\}$, and 
$H_{k} = \{a^k, ba^{k-1}, ca^{k-1}\}$ for $k \geq 2$. Then $I = \sum\limits_{k=0}^\infty P_k$, where $P_k$ is the projection onto $H_k$.

In this matrix decomposition, $L_a$, $L_b$, and $L_c$ are represented by 
\[ L_a = \begin{bmatrix}
0 & 0 & 0 & \dots \\
A_1 & 0  &  0 & \dots \\
0 & A  & 0 & \dots \\
0 & 0 & A & \dots \\
\vdots & \vdots & \vdots & \ddots \end{bmatrix}, \ \ 
 L_b = \begin{bmatrix}
0 & 0 & 0 & \dots \\
B_1 & 0  & 0 & \dots \\
0 & B  & 0 & \dots \\
0 & 0 & B & \dots \\
\vdots & \vdots & \vdots & \ddots \end{bmatrix}, 
L_c = \begin{bmatrix}
0 & 0 & 0 &  \dots \\
C_1 & 0  & 0 &  \dots \\
0 & C  & 0 &  \dots \\
0 & 0 & C &  \dots \\
\vdots & \vdots & \vdots & \ddots \end{bmatrix} \]
where $A_1 = \begin{bmatrix} 1 \\
0\\ 0 \end{bmatrix}, 
B_1 = \begin{bmatrix} 0 \\
1 \\ 
0 \end{bmatrix}$, $C_1 =  \begin{bmatrix} 0 \\
0 \\
1 \end{bmatrix}$, and $A, B$, and $C$ are defined as 
\[ A = \begin{bmatrix} 1 & 0 & 0 \\
0 & 0 & 1 \\
0 & 1 & 0  \end{bmatrix}, B = \begin{bmatrix} 0 & 1 & 0 \\
1 & 0 & 0 \\
0 & 0 & 1 \end{bmatrix}, C =  \begin{bmatrix} 0 & 0 & 1 \\
0 & 1 & 0 \\
1 & 0 & 0 \end{bmatrix}.\]

Notice that this free semigroupoid algebra has a non-zero nilpotent element $T = L_a + \omega L_b + \omega^2 L_c$,
where $\omega$ is a primitive third root of unity. This differs from the higher rank graph case, since left regular free semigroupoid algebras from single-vertex higher rank graphs have no non-zero nilpotents (\cite{thesis}, Proposition 3.4.2).
\end{example}

To study semisimplicity in the category of paths case, it will be helpful to employ the definition of a degree functor from \cite{spielberg} (restricted here to functions mapping into $\mathbb{N}^k)$:

\begin{definition} 
A \emph{degree functor} on $\Lambda$ is a function $\varphi$ from $\Lambda$ to $\mathbb{N}^k$ such that for all $\mu, \nu \in \Lambda$ satisfying $s(\mu) = r(\nu)$:
\[ \varphi(\mu\nu) = \varphi(\mu) + \varphi(\nu).\]
We say the degree functor is \emph{non-degenerate} if $\varphi(\alpha) \neq 0$ when $\alpha \notin \Lambda^0$. Let $|\mu| = |\varphi(\mu)|$.
\end{definition}

Moreover, we propose the following definitions:

\begin{definition}
Let $\Lambda$ be a category of paths with a degree functor. A path $\mu \in \Lambda$ is called \textit{minimal} if for $\nu, \eta \in \Lambda$, $\mu = \nu \eta$ implies $\mu = \nu$ or $\mu = \eta$.
\end{definition}

\begin{definition} Say that a category of paths $\Lambda$ satisfies property (P) if:
\begin{itemize}
\item[(i)] For each vertex $v \in \Lambda^0$, the set of minimal paths in $v\Lambda$ is finite; and
\item[(ii)] If $A \neq 0$ and $A = a_1 L_{w_1} + a_2 L_{w_2} + \dots + a_k L_{w_k}$ where $|w_1| = |w_2| = \dots = |w_k|$, then there is some $\mu \in \Lambda$ such that $L_\mu A$ is not nilpotent.
\end{itemize}
\end{definition}

This condition has two parts: the first is similar to row-finiteness in a graph; the second is a restriction on which elements of the algebra can be nilpotent, which is similar to, but more general than, the requirement that all paths lie on a cycle. The following can be shown:
\begin{theorem} (\cite{thesis}, Theorem 4.1.9) Let $\Lambda$ be a category of paths with a non-degenerate degree functor. Then:
\begin{enumerate}
\item If $\Lambda$ satisfies (P), then $\mathfrak{L}_\Lambda$ is semisimple.
\item If $\mathfrak{L}_\Lambda$ is semisimple, then each path in $\Lambda$ lies on a cycle.
\end{enumerate}
\end{theorem}

An example of a category of paths that satisfies (P) is Example \ref{3loop2} (\cite{thesis}, Corollary 4.2.4). Therefore, $\mathfrak{L}_\Lambda$ is semisimple for this example.

As for reflexivity, the proof of Theorem 6.4 from \cite{kribs.power.2006} also proves the following proposition. We say $v$ is a \textit{radiating vertex} if for all $\lambda \in \Lambda$, $r(\lambda) = v$ implies $s(\lambda) = v$.

\begin{proposition}\label{reflexive conditions}
Suppose that $\Lambda$ is a category of paths with a non-degenerate degree functor such that each radiating vertex $v$ satisfies
\begin{enumerate}
\item[(a)] for the category of paths $\Lambda'$ consisting of $v$ and all paths $\mu \in \Lambda$ with $r(\mu) = s(\mu) = v$, we have $\mathfrak{L}_{\Lambda'}$ is reflexive; and
\item[(b)] if $\mu_1$ and $\mu_2$ are loops at $v$ and $w_1$ and $w_2$ are paths with source $v$, then $w_1 \mu_1 \neq w_2 \mu_2$.
\end{enumerate}
Then $\mathfrak{L}_\Lambda$ is reflexive.
\end{proposition}

This suggests that the place to start with reflexivity is to determine which single-vertex categories of paths have reflexive free semigroupoid algebras. Then we can use that information to analyze the reflexivity of multiple-vertex categories of paths.


\bibliographystyle{acm}
\bibliography{refer.bib}

\end{document}